\documentclass[11pt]{article}
\usepackage{amsmath, amssymb, latexsym, amsfonts, amscd, amsthm, color, hyperref, mathabx}
\usepackage{graphicx}
\usepackage[OT2,T1]{fontenc}
\DeclareSymbolFont{cyrletters}{OT2}{wncyr}{m}{n}
\DeclareMathSymbol{\Sha}{\mathalpha}{cyrletters}{"58}
\newtheorem{theorem}{Theorem}

\newtheorem{lemma}{Lemma}
\newtheorem{corollary}{Corollary}

\newtheorem{example}{Example}
\newtheorem{conjecture}{Conjecture}
\newtheorem{question}{Question}

\DeclareMathOperator{\R}{\mathbb{R}}

\renewcommand{\phi}{\varphi}

\title{A Proof of the HRT Conjecture for Widely Spaced Sets}
\author{Michael Kreisel \\ \href{mailto:michael.c.kreisel@gmail.com}{michael.c.kreisel@gmail.com}}
\begin{document}
\date{}
\maketitle

\begin{abstract}
Given $f \in C_0(\R^n)$ and $\Lambda \subset \R^{2n}$ a finite set we demonstrate the linear independence of the set of time-frequency translates $\mathcal{G}(f, \Lambda) = \{\pi(\lambda)f\}_{\lambda\in \Lambda}$ when the time coordinates of points in $\Lambda$ are far apart relative to the decay of $f.$ As a corollary, we prove that for any $f \in C_0(\R^n)$ and finite $\Lambda \subset \R^{2n}$ there exist infinitely many dilations $D_r$ such that $\mathcal{G}(D_rf, \Lambda)$ is linearly independent. Furthermore, we prove that $\mathcal{G}(f, \Lambda)$ is linearly independent for functions like $f(t) = \frac{cos(t)}{|t|}$ which have a singularity and are bounded away from any neighborhood of the singularity.
\end{abstract}

\begin{section}{Introduction}
Consider the translation operator $T_x f = f(t-x)$ and the modulation operator $M_\omega f = e^{2 \pi i \omega \cdot t} f(t)$ acting on $f \in L^2(\R^n).$ For $\lambda = (x, \omega) \in \R^{2n}$ we define the time-frequency shift $\pi(\lambda)f = M_\omega T_x f.$ The Heil-Ramanathan-Topiwala (HRT) Conjecture \cite{HRT} states:
\begin{conjecture}
Suppose $f \in L^2(\R)$ is nonzero and $\Lambda \subset \R^{2}$ is a finite set. Then the collection of functions $\mathcal{G}(f, \Lambda) = \{\pi(\lambda)f\}_{\lambda \in \Lambda}$ is linearly independent. 
\end{conjecture}

\noindent The HRT Conjecture is still open in its most general form, but it has been proven under various additional assumptions on the function $f$ and the point set $\Lambda$ \cite{BenBou}, \cite{BowSpe}, \cite{BowSpe2}, \cite{Dem}, \cite{DemZah}, \cite{HRT}, \cite{Lin}, \cite{Oko}. 

In this paper we will prove Conjecture 1 in cases where the distance between points in $\Lambda$ is large relative to the decay of $f.$ For example, we will prove the following theorem. We denote by $C_0(\R^n)$ the space of continuous functions for which $|f(t)| \rightarrow 0$ as $||t|| \rightarrow \infty.$
\begin{theorem}
Let $f \in C_0(\R^n),$ let $\Lambda = \{(x_i,\omega_i)\}_{i=1}^N \subset \R^{2n}$ be a finite set, and fix $R$ so that $|f(t)| < \frac{|f(0)|}{N-1}$ for all $t$ outside of the ball of radius $R$ around the origin. If $||x_i - x_j|| > R$ whenever $ i \neq j$ then $\mathcal{G}(f,\Lambda)$ is linearly independent. 
\end{theorem}
\noindent
The intuition for Theorem 1 is that if the points in $\Lambda$ are spaced far apart then in any linear dependence the tails of translates of $f$ must combine to cancel the peaks of $f.$ This requires putting large coefficients on the translates of $f.$ However by putting large coefficients on the translates of $f,$ we make the peaks of the translates more difficult to cancel, leading to a contradiction. 

Theorem 1 is most effective when $|f|$ drops off steeply near the origin. Given any $f \in C_0(\R^n)$ we can engineer such a steep descent by applying a sufficiently large dilation, assuming $f(0) \neq 0.$ We denote by $D_r$ the unitary operator which dilates a function $f$ uniformly along all the coordinate axes:

$$D_rf = |r|^{\frac{n}{2}} f(rt_1, \dots, rt_n).$$

\begin{corollary}
Suppose $f \in C_0(\R^n)$ and $f(0) \neq 0.$ Given $\Lambda = \{(x_i,\omega_i)\}_{i=1}^N \subset \R^{2n}$ there exists $r > 0$ such that $\mathcal{G}(D_{r'}f, \Lambda)$ is linearly independent for all $0 < r' < r.$ 
\end{corollary}

Similarly, if $f$ has a singularity away from which it is bounded then we can find translates of $f$ which have an arbitrarily steep drop off. Thus we can prove Conjecture 1 for such functions.

\begin{theorem}
Let $f$ be continuous except at a point $p$ where $\lim_{t \rightarrow p} |f(t)| = \infty.$ Assume that $f$ is bounded away from any neighborhood of $p.$ Then $\mathcal{G}(f, \Lambda)$ is linearly independent for any finite $\Lambda \subset \R^{2n}.$
\end{theorem} 

\noindent In the next section we will prove Theorems 1 and 2 along with some further variations on the same theme. The spirit of our analysis is most similar to the extension principle in \cite{Oko}, however our strategy and results are different. 

\end{section}

\begin{section}{Proofs, Examples, and Extensions}
The following lemma captures the intuition for Theorem 1 described above:
\begin{lemma}
Let $S = \{x_i\}_{i=1}^N \subset \R^n, f \in C(\R^n),$ and $E = \{e_i\}_{i=1}^N \subset C(\R^n)$ such that $|e_i(t)| = 1$ for all $t \in \R^n.$ Furthermore, suppose that $|f(x_i - x_j)| < \frac{|f(0)|}{N-1}$ whenever $x_i, x_j$ are distinct points in $S.$ Then the collection of functions $\{e_i(t)f(t-x_i)\}_{i=1}^N$ is linearly independent. 
\end{lemma}
\begin{proof}
Assume that the functions  $\{e_i(t)f(t-x_i)\}_{i=1}^N$ are linearly dependent, so that for some coefficients $\{c_i\}_{i=1}^N$ we have $$\sum_{i=1}^N c_ie_if(t-x_i)= 0.$$ Since $f$ is continuous this equality holds for all $t \in \R^n.$ If we evaluate the LHS at the point $x_j$ we can rearrange to get the following inequality:
\begin{align*}
|c_j||f(0)| &= \left |\sum_{i=1, i \neq j}^Nc_ie_i(x_j)f(x_j-x_i) \right |\\
&\leq \sum_{i=1, i \neq j}^N|c_i||f(x_j-x_i)|\\
&<\frac{|f(0)|}{N-1}\sum_{i=1, i \neq j}^N|c_i|.
\end{align*}
After summing all these inequalities and canceling $|f(0)|$ from each side, we see that
$$\sum_{j=1}^N |c_j| < \frac{1}{N-1}\sum_{j=1}^N\sum_{i=1, i \neq j}^N|c_i| = \sum_{j=1}^N |c_j|$$
which is a contradiction. The last equality follows since each term $|c_i|$ appears in exactly $N-1$ of the inner sums in the second expression. 
\end{proof}

Now we can prove our first theorem:
\begin{proof}[Proof of Theorem 1]
We can apply Lemma 1 with $S=\{x_i\}_{i=1}^N$ and $E = \{e^{2\pi i \omega_i \cdot t}\}_{i=1}^N.$ Since the points $x_i-x_j$ all lie outside the ball of radius $R$ around the origin, $|f(x_i - x_j)| < \frac{|f(0)|}{N-1}$ as required.
\end{proof}

\noindent Note that we only need to assume that the time coordinates of the points in $\Lambda$ are spaced far apart for Theorem 1 to hold. Although the specific value $f(0)$ suspiciously appears in our hypothesis, $\mathcal{G}(f, \Lambda)$ is linearly independent if and only if $\mathcal{G}(T_xf, \Lambda)$ is linearly independent for all $x \in \R^n,$ so we can always translate $f$ to put the most advantageous value at the origin.

Given Theorem 1, it is straightforward to deduce Corollary 1:

\begin{proof}[Proof of Corollary 1]
Since $f \in C_0(\R^n)$ and $f(0) \neq 0,$ we can find a value $R>0$ such that $|f(t)| < \frac{|f(0)|}{N-1}$ for all $t$ outside of a ball of radius $R$ around $0.$ Applying a dilation $D_r,$ we see that $|D_rf(t)| < \frac{|D_rf(0)|}{N-1}$ whenever $t$ lies outside a ball of radius $rR.$ Let $M = \min_{i,j}||x_i-x_j||$ be the minimum distance between any two points in $\Lambda.$ Then whenever $0 < r < \frac{M}{R}$ we can apply Theorem 1 to show that $\mathcal{G}(D_rf, \Lambda)$ is linearly independent. 
\end{proof}

Since translations and modulations are exchanged under the Fourier transform, we get an analogous result in the frequency domain. We denote by $\hat{f}$ the Fourier transform of $f.$

\begin{corollary}
Let $f \in L^1(\R^n)$ so that $\hat{f} \in C_0(\R^n).$ Let $\Lambda = \{(x_i,\omega_i)\}_{i=1}^N \subset \R^{2n}$ and fix $R$ so that $|\hat{f}(\omega)| < \frac{|\hat{f}(0)|}{N-1}$ for all $\omega$ outside of the ball of radius $R$ around the origin. If $||\omega_i - \omega_j|| > R$ whenever $ i \neq j$ then $\mathcal{G}(f,\Lambda)$ is linearly independent. 
\end{corollary}
\begin{proof}
First note that $\widehat{T_x f}$ = $M_{-x}\hat{f}$ and $\widehat{M_\omega f} = T_\omega \hat{f}.$ Therefore we can define $$\Lambda' = \{(\omega_i, -x_i) \,|\, (x_i, \omega_i) \in \Lambda\}$$  and $\mathcal{G}(f, \Lambda)$ is linearly independent if and only if $\mathcal{G}(\hat{f}, \Lambda')$ is linearly independent. However Theorem 1 immediately applies to show $\mathcal{G}(\hat{f}, \Lambda')$ is linearly independent.
\end{proof}

We have a similar extension of Corollary 1 to the frequency domain.
\begin{corollary}
Let $f \in L^1(\R^n)$ so that $\hat{f} \in C_0(\R^n)$ and let $\Lambda = \{(x_i,\omega_i)\}_{i=1}^N \subset \R^{2n}.$ Then there exists a value $r > 0$ such that $\mathcal{G}(D_{r'}f, \Lambda)$ is linearly independent whenever $r' > r.$
\end{corollary}
\begin{proof}
The proof is the same as for Corollary 1, after noting that $\widehat{D_rf} = D_{\frac{1}{r}}\hat{f}$ and that the Fourier transform rotates the time-frequency plane as described in the proof of Corollary 2.
\end{proof}

\begin{example}
Consider the family of functions
$$f_{C, \omega} = \begin{cases} 
       \frac{\cos(\omega t)}{|t|} & |t| \geq \frac{1}{C}\\
      C \cos(\omega t)& |t| < \frac{1}{C}
   \end{cases}$$
   
The functions $f_{C, \omega}$ are in $L^2(\R) \cap C_0(\R).$ Nonetheless they decay slowly at infinity and  oscillate in the tail. To the author's knowledge, such functions are not covered by the results of \cite{BenBou}, \cite{BowSpe}, or \cite{Oko} which assume fast decay at infinity or ultimate positivity. Given $\Lambda = \{(x_i,\omega_i)\}_{i=1}^N$ let $M = \min_{i,j}|x_i-x_j|.$ Then by applying Theorem 1, we can see that $\mathcal{G}(f_{C, \omega}, \Lambda)$ is linearly independent whenever $C > \frac{N-1}{M}.$ For the four point set $\Lambda' = \{(0,0), (1,0), (0,1), (\sqrt{2}, \sqrt{2})\}$ we have $\mathcal{G}(f_{\omega, C}, \Lambda')$ linearly independent whenever $C > \frac{3}{\sqrt{2} - 1}.$
\end{example}

Example 1 suggests that the function $f(t) = \frac{\cos(\omega t)}{|t|}$ should satisfy Conjecture 1 in full, as it is the pointwise limit of $f_{C,\omega}$ as $C \rightarrow \infty.$ This is true, and is implied by Theorem 2 which we are now ready to prove.

\begin{proof}[Proof of Theorem 2]
Without loss of generality we may assume that $p=0,$ since we can always translate $f$ to place the singularity at the origin. If we fix $\Lambda \subset \R^{2n}$ of size $N$ such that the minimum distance between the $x-$coordinates in $\Lambda$ is $R,$ we would like to find a translate of $f$ which satisfies $|f(t+x)| < \frac{|f(x)|}{N-1}$ outside the ball of radius $R$ around $x.$ If we can find such an $x$ then the argument in the proof of Lemma 1 applies to show $\mathcal{G}(f, \Lambda)$ is linearly independent. To find such an $x,$ we first note that since $f$ is bounded away from $0$ we can find $A$ such that $|f(t)| < A$ outside a ball of radius $\frac{R}{2}$ around $0.$ Since $\lim_{t \rightarrow 0} |f(t)| = \infty,$ we can find an $x$ less than $\frac{R}{2}$ such that  $|f(x)| > A(N-1),$ and this $x$ satisfies the criteria described above.
\end{proof}

\begin{example}
We can adapt the examples above to find functions in $L^2(\R)$ satisfying Conjecture 1. Consider the family of functions
$$g_\omega(t) = \begin{cases} 
      \frac{\cos(\omega t)}{|t|^{\frac{1}{4}}} & |t| < 1 \\
      \frac{\cos(\omega t)}{|t|} & otherwise
   \end{cases}$$
The functions $g_\omega(t)$ are in $L^2(\R) \cap C_0(\R).$ Theorem 2 applies to show that they satisfy Conjecture 1.
\end{example}

The previous results apply when all points in $\Lambda$ are sufficiently far apart in either time or frequency. By applying the Short Time Fourier Transform (STFT) we can demonstrate linear independence when the points in $\Lambda$ are sufficiently far apart in the time-frequency plane. For $f,g \in L^2(\R^n)$ the STFT of $f$ with respect to $g$ is given by
$$V_gf(\lambda) = \langle f, \pi(\lambda)g \rangle.$$ It is easy to see \cite{Gro} that $V_gf \in C_0(\R^{2n})$ and satisfies the identity $$V_g T_{u}M_{\eta}f(x, \omega) = e^{-2 \pi i u \cdot \omega}V_gf(x-u, \omega - \eta).$$
\begin{theorem}
Suppose $f,g \in L^2(\R^n)$ so that $V_gf \in C_0(\R^{2n}).$ Let $\Lambda = \{\lambda_i\}_{i=1}^N \subset \R^{2n}$ and fix $R$ so that $|V_gf(\lambda)| < \frac{|V_gf(0)|}{N-1} = \frac{|\langle f, g \rangle|}{N-1}$ for all $\lambda$ outside of the ball of radius $R$ around the origin. If $||\lambda_i - \lambda_j|| > R$ whenever $ i \neq j$ then $\mathcal{G}(f,\Lambda)$ is linearly independent. 
\end{theorem}
\begin{proof}
Suppose $\mathcal{G}(f, \Lambda)$ is linearly dependent so that for some coefficients $c_i$ we have
$$\sum_{i=1}^N c_i\pi(\lambda_i)f = 0.$$
Then by applying the STFT with respect to $g$ we have
$$\sum_{i=1}^N c_i'e^{-2 \pi i u_i \cdot \omega}V_gf(\lambda - \lambda_i)  = \sum_{i=1}^N c_i'V_g\pi(\lambda_i)f = 0$$
where $u_i$ denotes the time coordinate of $\lambda_i$ and $c_i' = e^{-2\pi i x_i \omega_i}c_i.$ However we can apply Lemma 1 to $V_gf$ with $S = \{\lambda_i\}_{i=1}^N$ and $E = \{e^{-2 \pi i u_i \cdot \omega}\}_{i=1}^N$ to show that the functions $\{e^{-2 \pi i u_i \cdot \omega}V_gf(\lambda - \lambda_i)\}_{i=1}^N$ must be linearly independent, which is a contradiction. 
\end{proof}

\end{section}

\begin{section}{Discussion}
Our Lemma 1 and Theorem 1 demonstrate that $\mathcal{G}(f, \Lambda)$ is linearly independent when the points of $\Lambda$ are far apart relative to the decay in $f.$ However our proofs use no properties specific to the modulations $e^{2 \pi i \omega t},$ and apply just as well to functions in $L^p(\R^n)$ when $n>1$ and $p>2.$ Given the generality of Theorem 1 and in light of the following example, we can see that Theorem 1 alone provides only loose evidence for Conjecture 1. It would need to be combined with tools more specific to Conjecture 1 if we hoped to use it to make further progress. 
\begin{example}
In \cite{EdgRos} the authors demonstrate that the function
$$f(a,b) = \int_{\frac{1}{3}}^{\frac{2}{3}} exp(i(a\cos^{-1}(t) + b\cos^{-1}(1-t)))dt$$
is in $C_0(\R^2) \cap L^p(\R^2)$ for $p > 4$ and satisfies the dependence
$$2f(a,b) = f(a+1, b) + f(a-1, b) + f(a, b+1) + f(a, b-1).$$
Nonetheless, our Theorem 1 and Corollary 1 can be applied to $f,$ though Theorem 1 clearly does not rule out the dependence above. 
\end{example}

To this end, we can try to apply \emph{metaplectic transforms} in an attempt to satisfy the hypotheses of Theorem 1 in more cases. For example, for $f \in L^2(\R)$ and $(x, \omega) \in \R^2$ the dilation operator $D_r$ satisfies
$$D_r M_\omega T_x f = M_{\frac{\omega}{r}}T_{xr}D_r f.$$
Given a point set $\Lambda$ we can apply a large dilation to move the points far apart, but we also stretch the function $f$ so that it decays very slowly away from the origin. Since the function $f$ stretches at the same rate that the points are moved apart, $\mathcal{G}(f, \Lambda)$ satisfies the hypotheses of Theorem 1 if and only if $\mathcal{G}(D_rf, D_r\Lambda)$ does, where $D_r\Lambda = \{(rx, \frac{\omega}{r})\}_{(x, \omega) \in \Lambda}.$ Thus dilations are no help in furthering the uses of Theorem 1 for $f.$ 

Another useful metaplectic transform is modulation by a linear chirp. If $S_rf(t) = e^{2\pi i r t^2 f(t)}$ then
$$S_r M_\omega T_x f = M_{\omega - xr}T_x S_rf.$$
Similarly, if we let $U_r = \widecheck{(e^{2 \pi i r t^2})}*f(t)$ then
$$U_rM_\omega T_x f = M_{-\omega}T_{-x-r\omega} U_rf. $$
For the purposes of applying our theorems, the operators $S_r$ and $U_r$ look promising, as compositions of them can be used to increase the distance between points in the time-frequency plane. However, as in the case of dilations, we must understand how compositions of $S_r$ and $U_r$ affect the decay of $f$ and compare this with the increase in distance between points. The author has not made substantive progress to demonstrate new examples by applying these transforms, nor has he been able to rule out their utility.

In another direction, one could try to expand the utility of Theorem 3 by leveraging the choice of window function as a free variable. One could leave $f$ and $\Lambda$ fixed but vary the window function $g$ in an attempt to satisfy the hypotheses. This leads naturally to the question:
\begin{question}
Given $f \in L^2(\R^n), R > 0, N > 0$ can we design a window $g \in L^2(\R^n)$ so that $|V_gf| < \frac{|\langle f, g \rangle|}{N}$ outside the ball of radius $R$ around the origin?
\end{question}
\noindent A positive answer to Question 1 would prove the HRT conjecture. We would want to design $g$ so that $V_gf$ decreases sharply near the origin and then has a fat tail, since we know that the probability mass of $V_gf$ cannot be too heavily concentrated near the origin due to various uncertainty principles for the STFT. Alternatively, it may be possible to develop a kind of uncertainty principle which answers Question 1 negatively. 
\end{section}

\subsection*{Acknowledgements}
The author would like to thank Radu Balan and Kasso Okoudjou for introducing him to the HRT Conjecture and for helpful discussions about this work. The author would also like to thank an anonymous reviewer for encouraging him to expand the results contained in earlier drafts. 

\bibliographystyle{plain}
\bibliography{bibtex}

\begin{thebibliography}{10}

\bibitem{BenBou}
J.~J. Benedetto and A.~Bourouihiya.
\newblock Linear independence of finite gabor systems determined by behavior at
  infinity.
\newblock {\em J. Geom. Anal.}, page 226–254, 2015.

\bibitem{BowSpe}
M.~Bownik and D.~Speegle.
\newblock Linear independence of time-frequency translates of functions with
  faster than exponential decay.
\newblock {\em Bull. Lond. Math. Soc.}, page 554–566, 2013.

\bibitem{BowSpe2}
M.~Bownik and D.~Speegle.
\newblock Linear independence of time-frequency translates in $\mathbb{R}^d$.
\newblock {\em J. Geom. Anal.}, page 1678–1692, 2016.

\bibitem{Dem}
C.~Demeter.
\newblock Linear independence of time frequency translates for special
  configurations.
\newblock {\em Math. Res. Lett.}, 17:761–779, 2010.

\bibitem{DemZah}
C.~Demeter and A.~Zaharescu.
\newblock Proof of the {H}{R}{T} conjecture for (2,2) configurations,.
\newblock {\em J. Math. Anal. Appl.}, 388:151–159, 2012.

\bibitem{EdgRos}
G.~Edgar and J.~Rosenblatt.
\newblock Difference equations over locally compact abelian groups.
\newblock {\em Trans. {A}mer. {M}ath. {S}oc.}, 253:273--289, 1979.

\bibitem{Gro}
K.~Gr{\"o}chenig.
\newblock {\em Foundations of Time-Frequency Analysis}.
\newblock Birkh$\ddot{\text{a}}$user, 2001.

\bibitem{HRT}
C.~{H}eil, {J}. {R}amanathan, and {P}. {T}opiwala.
\newblock Linear independence of time-frequency translates.
\newblock {\em Proc. Amer. Math Soc.}, 124:2787--2795, 1996.

\bibitem{Lin}
P.A. Linnell.
\newblock von {N}eumann algebras and linear independence of translates.
\newblock {\em Proc. Amer. Math. Soc.}, 127:3269–3277, 1999.

\bibitem{Oko}
K.~A. {Okoudjou}.
\newblock {Extension and restriction principles for the HRT conjecture}.
\newblock {\em ArXiv e-prints}, January 2017.

\end{thebibliography}

\end{document}